\documentclass[11pt,reqno]{amsart}

\usepackage{amssymb}
\usepackage{amsthm}
\usepackage{amsmath}

\newtheorem{thm}{Theorem}[section]
\newtheorem{prop}[thm]{Proposition}
\newtheorem{lem}[thm]{Lemma}
\newtheorem{cor}[thm]{Corollary}

\newcommand{\cD}{{\mathcal{D}}}

\newcommand{\TT}{\mathbb{T}}
\newcommand{\DD}{\mathbb{D}}
\newcommand{\RR}{\mathbb{R}}

\renewcommand{\hat}{\widehat}

\DeclareMathOperator{\dist}{dist}

\begin{document}

\title{Cantor sets and cyclicity in  weighted Dirichlet spaces }

\author[El--Fallah]{O. El-Fallah$^1$}
\address{D\'{e}partement de Math\'ematiques\\ Universit\'e Mohamed V\\ B.P. 1014 Rabat\\ Morocco }
\email{elfallah@fsr.ac.ma} 
\thanks{{1. Research partially supported by  a grant from Egide Volubilis  (MA09209)}}

\author[Kellay]{K. Kellay$^2$}
\address{CMI\\ LATP\\ Universit\'e de Provence\\39 rue F. Joliot-Curie\\ 13453 Marseille\\Ê France }
\email{kellay@cmi.univ-mrs.fr}
\thanks{{2. Research partially supported by grants from Egide Volubilis  (MA09209) and ANR Dynop}}

\author[Ransford]{T. Ransford$^3$}
\address{D\'epartement de math\'ematiques et de statistique\\ Universit\'e Laval\\ Qu\'ebec (QC)\\ Canada G1V 0A6}
\email{ransford@mat.ulaval.ca}
\thanks{{3. Research partially supported by grants from NSERC (Canada), FQRNT (Qu\'ebec) and the Canada Research Chairs program}}

\begin{abstract}
We treat  the problem of characterizing the cyclic vectors in the weighted Dirichlet spaces,
extending some of our earlier results in the classical Dirichlet space.
The  absence of a Carleson-type formula for weighted Dirichlet integrals necessitates the introduction of new techniques. 
\end{abstract}


\maketitle

\section{Introduction}\label{S:intro}
In this paper we study the weighted Dirichlet spaces $\cD_\alpha~(0\le\alpha\le1)$,  
defined by
$$
\cD_\alpha:=
\Bigl\{f\in\text{hol}(\DD): \cD_\alpha(f):=
\frac{1}{\pi} \int_\DD |f'(z)|^2(1-|z|^2)^\alpha\,dA(z)<\infty\Bigr\}.
$$
Here $\DD$ denotes the open unit disk,  and $dA$ is area measure on $\DD$.
Clearly $\cD_\alpha$ is a  Hilbert space with respect to the norm $\|\cdot\|_\alpha$ given by
$$
\|f\|_{\alpha}^2:= |f(0)|^2+\cD_\alpha(f).
$$
A classical calculation shows that, if $f(z)=\sum_{n\ge0}a_nz^n$,
then
$$ \|f\|_\alpha^2 \asymp\sum_{n\ge0}(n+1)^{1-\alpha}|a_n|^2.$$
Note that $\cD_1=H^2$ is the usual Hardy space, and $\cD_0$ is the classical Dirichlet space (thus our labelling convention follows  \cite{Al} rather than  \cite{BS}).

An {\em invariant subspace} of $\cD _\alpha$
is a closed subspace $M$ of $\cD _\alpha$ such that $ zM\subset M$.
Given $f\in \cD _\alpha$, we denote by $[f]_{\cD _\alpha}$ the smallest invariant
subspace of $\cD _\alpha$ containing $f$, namely
the closure in $\cD_\alpha$ of $\{pf: p \text{ a polynomial}\}$.
We say that $f$ is {\em cyclic} for $\cD _\alpha$ if $[f]_{\cD _\alpha}=\cD _\alpha$.
The survey article \cite{EKR2} gives a brief history of invariant subspaces and cyclic functions in the classical case $\alpha=0$.

Our goal is to characterize the cyclic functions of $\cD_\alpha$.
In order to state our  results, we introduce the notion of $\alpha$-capacity. 
For $\alpha\in[0,1)$, we define the kernel function $k_{\alpha}:\RR^+\to\RR\cup\{\infty\}$ by
$$
k_{\alpha}(t):=
\begin{cases}
1/t^\alpha, & 0<\alpha<1,\\
\log(1/ t),& \alpha =0.
\end{cases}
$$
The \emph{$\alpha$-energy} of a (Borel) probability measure $\mu$ on $\TT$ is defined by
 $$
 I_\alpha(\mu):=
 \iint k_{\alpha}(|\zeta-\zeta'|)\,d\mu(\zeta)\,d\mu(\zeta').
 $$
 A standard calculation gives
$$ 
I_\alpha(\mu) \asymp \sum _{n\ge 0}\frac{|\hat {\mu}(n)|^2}{(1+n)^{1-\alpha}}.
 $$
The \emph{$\alpha$-capacity} of a  Borel subset $E$ of $\TT$ is defined by
 $$
 C_{\alpha}(E):=1/\inf\{I_\alpha(\mu): \mathcal{P}(E)\},
 $$
where $\mathcal{P}(E)$ denotes the set of all probability measures supported on compact subsets of $E$.  In particular, $C_\alpha(E)>0$ if and only if there exists a probability measure~$\mu$ supported on a compact subset of $E$ and having finite $\alpha$-energy.
If $\alpha=0$, then $C_0$ is the classical logarithmic capacity. 

We recall a result due to Beurling and Salem--Zygmund \cite[\S V, Theorem~3]{Ca2} 
about radial limits of functions in the weighted  Dirichlet spaces.  
If  $f\in \cD_\alpha$, then   $f^*(\zeta):=\lim\limits_{r\to 1-}f(r\zeta) $ exists
for all $\zeta\in\TT$ outside a set of $\alpha$-capacity zero.

The following theorem gives two necessary conditions for cyclicity in $\cD_\alpha$.

\begin{thm}\label{T:nec}
Let $\alpha\in[0,1)$. If $f$ is cyclic in $\cD_\alpha$, then 
\begin{itemize}
\setlength{\itemsep}{0pt}
\item $f$ is an outer function, 
\item $\{\zeta\in\TT:f^*(\zeta)=0\}$ is a set of $\alpha$-capacity zero.
\end{itemize}
\end{thm}

The first part is  \cite[Corollary~1]{BS}.
For $\alpha=0$, the second part is  \cite[Theorem~5]{BS},
and for general $\alpha$ the proof is similar, the only difference being that the logarithmic kernel $k_0$ is replaced by $k_\alpha$.
We omit the details.

Our main result is a partial converse to this theorem. 
To state it, we need to define the notion of a generalized Cantor set.

Let $(a_n)_{n\ge0}$ be a positive sequence such that $a_0\le2\pi$ and 
$$
\sup_{n\ge0} \frac{a_{n+1}}{a_n}< \frac{1}{2}.
$$
The \emph{generalized Cantor set} $E$ associated to  $(a_n)$ is constructed as follows.
Start with a closed arc of length $a_0$ on the unit circle $\TT$.
Remove an open arc from the middle,
to leave two closed arcs each of length $a_1$.
Then remove two open arcs from their middles to leave
four closed arcs each of length $a_2$. 
After $n$ steps,  we obtain $E_n$, the union of $2^{n}$ closed arcs  each of length $a_n$.
Finally, the generalized Cantor set is  $E:=\cap_n E_n$.

 \begin{thm}\label{T:suff} 
Let $\alpha\in[0,1)$ and let $f\in\cD_\alpha$. Suppose that:
\begin{itemize}
\setlength{\itemsep}{0pt}
\item $f$ is an outer function,
\item $|f|$ extends continuously to $\overline{\DD}$,
\item $\{\zeta\in\TT:|f(\zeta)|=0\}$ is contained in a generalized Cantor set $E$ of $\alpha$-capacity zero.
\end{itemize}
Then $f$ is cyclic  for $\cD_\alpha$.
\end{thm}

Functions $f$ satisfying the hypotheses exist in abundance. 
Indeed,  any generalized Cantor set 
is a so-called Carleson set,
and is thus the zero set of some  outer function $f$ such that $f$ and all 
its derivatives extend continuously to $\overline{\DD}$.
Moreover, it is very easy to determine which generalized Cantor sets 
have $\alpha$-capacity zero.
More details will be given in \S\ref{S:Cantor}.

To prove Theorem~\ref{T:suff}, we adopt the following strategy.  
In \S\ref{S:Korenblum}, using a technique due to Korenblum, 
we show that $[f]_{\cD_\alpha}$ 
contains at least those functions $g\in\cD_\alpha$ 
satisfying $|g(z)|\le \dist(z,E)^4$.
 The idea is then to take one simple such $g$,
and gradually transform it into the constant function $1$ while staying inside $[f]_{\cD_\alpha}$, thereby proving that $1\in[f]_{\cD_\alpha}$.
This requires three tools: a general estimate for weighted Dirichlet integrals of outer functions, some properties of generalized Cantor sets, and a regularization theorem. These tools are developed in \S\S\ref{S:Dint},\ref{S:Cantor},\ref{S:reg} respectively, and all the pieces are finally assembled in \S\ref{S:completion}, 
to complete the proof of Theorem~\ref{T:suff}.

Theorem~\ref{T:suff} was established for the classical Dirichlet space, $\alpha = 0$,
in \cite[Corollary~1.2]{EKR1}.  The proof there followed the same general strategy, but in several places  key use was made of a formula of Carleson \cite{Ca1} expressing the Dirichlet integral of  an outer function $f$  in terms of the values of  $|f^*|$ on the unit circle. No analogue of Carleson's formula is known in the case $0<\alpha<1$, 
and one of the main points of this note is to show how this difficulty may 
be overcome.

Throughout the paper, we use the notation $C(x_1,\dots,x_n)$ to denote a constant that depends only on $x_1,\dots,x_n$, where the $x_j$ may be numbers, functions or sets.
The constant may change from one line to the next.

\section{Korenblum's method}\label{S:Korenblum}

Our aim in this section is to prove the following theorem.

\begin{thm}\label{T:Korenblum} 
Let $f\in {\cD_\alpha}$ be an outer function such that 
$|f|$ extends continuously to $\overline{\DD}$, 
and let $F:=\{\zeta\in \TT: |f(\zeta)|=0\}$. 
If $g\in {\cD_\alpha}$ and
$$
|g(z)|\le \dist(z,F)^4 \quad (z\in\DD),
$$ 
then $g\in [f]_{\cD_\alpha}$.
\end{thm}

This theorem is a $\cD_\alpha$-analogue of \cite[Theorem~3.1]{EKR1}, 
which was proved using a technique of Korenblum.
We shall use the same basic technique here. However, the proof in \cite{EKR1} proceeded via a so-called fusion lemma, which, being based on Carleson's formula for the Dirichlet integral, is no longer available to us here. 
Its place is taken by Corollary~\ref{C:Korenblum} below.

We need to introduce some notation. 
Given an outer function $f$ and a Borel subset $\Gamma$ of $\TT$, we define 
$$
f_\Gamma (z):=\exp\Bigl( \frac{1}{2\pi}\int_{\Gamma}
\frac{\zeta+z}{\zeta-z}\log |f^*(\zeta)|\,|d\zeta|\Bigr) \quad(z\in\DD).
$$
We write $\partial\Gamma$ and $\Gamma^c$ for the boundary and complement of $\Gamma$ in $\TT$ respectively.

\begin{lem}\label{L:Korenblum} 
Let  $f$ be a bounded outer function. For every Borel set $\Gamma\subset\TT$,
$$
|f_\Gamma'(z)|\le C(f)(|f'(z)|+\dist(z,\partial\Gamma)^{-4}) \quad(z\in\DD).
$$
\end{lem}

\begin{proof}
Without loss of generality, we may suppose that $\|f\|_\infty\le1$.
Note that then $\|f_\Gamma\|_\infty\le1$ for all $\Gamma$.
Also, obviously, $\log|f^*|\le0$ a.e.\ on $\TT$, 
which will help simplify some of the calculations below.

We begin by observing that
$$
\frac{f_\Gamma'(z)}{f_\Gamma(z)}
=\frac{1}{2\pi}\int_\Gamma \frac{2\zeta}{(\zeta-z)^2}\log|f^*(\zeta)|\,|d\zeta| 
\quad(z\in\DD),
$$
from which it follows easily that
\begin{equation}\label{E:Gamma}
|f_\Gamma'(z)|\le \frac{2\log(1/|f(0)|)}{\dist(z,\Gamma)^2} \quad(z\in\DD).
\end{equation}

Our aim now is to prove a similar inequality, 
but with $\partial\Gamma$ in place of $\Gamma$. 
Set $G:=\{z\in\DD:\dist(z,\Gamma)\ge\dist(z,\Gamma^c)^2\}$. 
Clearly $\dist(z,\Gamma)\ge\dist(z,\partial\Gamma)^2$ for all $z\in G$,
so \eqref{E:Gamma} implies 
\begin{equation}\label{E:zinG}
|f_\Gamma'(z)|\le \frac{2\log(1/|f(0)|)}{\dist(z,\partial\Gamma)^4} \quad(z\in G).
\end{equation}
Now suppose that $z\in\DD\setminus G$.
Then $\dist(z,\Gamma^c)^2>\dist(z,\Gamma)\ge (1-|z|^2)/2$, and hence
$$
|f_{\Gamma^c}(z)|
=\exp\Bigl(\frac{1}{2\pi}\int_{\Gamma^c}\frac{1-|z|^2}{|\zeta-z|^2}\log|f^*(\zeta)|\,|d\zeta|\Bigr)
\ge|f(0)|^2.
$$
Since obviously $f_\Gamma=f/f_{\Gamma^c}$, it follows that, for all $z\in\DD\setminus G$,
$$
|f_\Gamma'(z)|
\le \frac{|f'(z)|}{|f_{\Gamma^c}(z)|}+\frac{|f(z)|}{|f_{\Gamma^c}(z)|^2}|f_{\Gamma^c}'(z)|
\le \frac{|f'(z)|}{|f(0)|^2}+\frac{1}{|f(0)|^4}\frac{2\log(1/|f(0)|)}{\dist(z,\Gamma^c)^2},
$$
where once again we have used \eqref{E:Gamma}, this time with $\Gamma$ replaced by $\Gamma^c$. Noting that $\dist(z,\Gamma^c)\ge \dist(z,\partial\Gamma)$ for all $z\in\DD\setminus G$, we deduce that
\begin{equation}\label{E:znotinG}
|f_\Gamma'(z)|\le \frac{|f'(z)|}{|f(0)|^2}+\frac{1}{|f(0)|^4}\frac{2\log(1/|f(0)|)}{\dist(z,\partial\Gamma)^2} \quad(z\in\DD\setminus G).
\end{equation}
The inequalities \eqref{E:zinG} and \eqref{E:znotinG} between them give the result.
 \end{proof}

 \begin{cor}\label{C:Korenblum}
 Let $\alpha\in[0,1)$ and $f\in\cD_\alpha\cap H^\infty$ be an outer function.
 Then, for every Borel set $\Gamma\subset\TT$ and every $g\in\cD_\alpha$ satisfying
 $|g(z)|\le\dist(z,\partial\Gamma)^4$, we have
 $$
 \|f_\Gamma g\|_\alpha\le C(\alpha,f)(1+\|g\|_\alpha).
 $$
 \end{cor}
 
 \begin{proof}
Using Lemma~\ref{L:Korenblum}, we have
 $$
 |(f_\Gamma g)'|\le |f_\Gamma'||g|+|f_\Gamma||g'|\le C(f)(|f'|+1+|g'|).
 $$
 The conclusion follows easily from this.
  \end{proof}

 \begin{proof}[Proof of Theorem~\ref{T:Korenblum}]
 Let $I$ be a connected component of $\TT \setminus F$,
say $I= (e^{ia},e^{ib})$. 
Let $\rho>1$, and define
\begin{align*}
\psi_\rho(z)&:=(z-1)^4/(z-\rho)^4,\\
\phi_\rho(z)&:= \psi_\rho(e^{-ia}z)\psi_\rho(e^{-ib}z).
\end{align*}
The first step is to show that $\phi_\rho f_{\TT\setminus I}\in[f]_{\cD_\alpha}$.
 
Let   $\epsilon >0$ and set $I_\epsilon:=(e^{i(a+\epsilon)},e^{i(b-\epsilon)})$ and
$$
\phi_{\rho,\epsilon}(z):=\psi_\rho(e^{-i(a+\epsilon)}z)\psi_\rho(e^{-i(b-\epsilon)}z).
$$
By Corollary~\ref{C:Korenblum},
\begin{equation}\label{E:phieps}
\| \phi_{\rho,\epsilon} f_{\TT\setminus I_\epsilon}\|_\alpha
\le C(f) \|\phi_{\rho,\epsilon}\|_\alpha \le C(f,\rho).
\end{equation}
Note that  $|f_{\TT\setminus I_\epsilon}|= |f|$ in a neighborhood of   $\TT \setminus I$. 
Since $|f|$ does not vanish inside $I$, it follows that 
$|\phi_{\rho,\epsilon}f_{\TT\setminus I_\epsilon}|/|f|$ is bounded on $\TT$. 
Also $f$ is an outer function. 
Therefore, by a theorem of  Aleman \cite[Lemma~3.1]{Al},
$$
\phi_{\rho ,\epsilon}f_{\TT\setminus I_\epsilon} \in [f]_{\cD_\alpha}.
$$
Using  \eqref{E:phieps}, we see that
$\phi_{\rho,\epsilon} f_{\TT\setminus I_\epsilon}$ converges weakly in $\cD_\alpha$ 
to $\phi_\rho f_{\TT\setminus I}$ as $\epsilon\to 0$. Hence 
$\phi_\rho f_{\TT\setminus I} \in [f]_{\cD_\alpha}$, as claimed.

Next, we multiply by $g$.
As $g\in\cD_\alpha\cap H^\infty$, Aleman's theorem immediately yields
$\phi_\rho f_{\TT\setminus I} g\in [f]_{\cD_\alpha}$.
Using the fact that $|g(z)|\le \dist(z,F)^4$, it is easy to check that $\|\phi_{\rho}g\| _{\alpha}$
remains bounded as $\rho\to1$. By Corollary~\ref{C:Korenblum} again,
 $\|\phi_\rho f_{\TT\setminus I}g\|_{\alpha}$ is uniformly bounded, 
and $\phi_\rho f_{\TT\setminus I}g$ converges weakly to $f_{\TT\setminus I}g$. Hence
 $f_{\TT\setminus I}g\in [f]_{\cD_\alpha}$.

Now let $(I_j)_{j\ge1}$ be  the complete set of components of $\TT\setminus F$,
and set $J_n:=\cup_1^nI_j$. An argument similar to that above gives
$f_{\TT\setminus J_n}g\in[f]_{\cD_\alpha}$ for all~$n$. Moreover,
$\|f_{\TT\setminus J_n}g\|_{\alpha}$ is uniformly bounded. Thus 
$f_{\TT\setminus J_n}g$ converges weakly to~$g$, 
and so finally $g\in[f]_{\cD_\alpha}$.
\end{proof}

\section{Estimates for weighted Dirichlet integrals}\label{S:Dint}
 
The following result will act as a partial substitute for Carleson's formula.
  
\begin{thm}\label{T:Carleson}
Let $\alpha\in[0,1)$, and let $h:\TT\to\RR$ be a positive measurable function such that, for every arc $I\subset\TT$, 
\begin{equation}\label{E:hjensen}
\frac{1}{|I|}\int_I h(\zeta)\,|d\zeta|
\ge |I|^{\alpha}.
\end{equation}
If  $f$ is an outer function, then
$$
\cD_\alpha(f)\le
\frac{1}{\pi}\iint_{\TT^2}
\frac{(|f^*(\zeta)|^2-|f^*(\zeta')|^2)(\log|f^*(\zeta)|-\log|f^*(\zeta')|)}
{|\zeta-\zeta'|^{2}} (h(\zeta)+ h(\zeta'))|d\zeta||d\zeta'|.
$$
\end{thm}

\begin{proof}
Given  $z\in\DD$,  let  $I_z$ be the arc of $\TT$ 
with midpoint  $z/|z|$ and  
arclength $|I_z|=2(1-|z|^2)$. 
For  $\zeta\in I_z$, we have $|\zeta-z|\le 2(1-|z|^2)$,
and so 
$$
P(z,\zeta):=\frac{1-|z|^2}{|\zeta-z|^2}\ge \frac{1}{4(1-|z|^2)}=\frac{1}{2|I_z|}.
$$
Hence, using \eqref{E:hjensen}, we have
$$
\int_\TT P(z,\zeta) h(\zeta)\, |d\zeta|
\ge \frac{1}{2|I_z|}\int_{I_z} h (\zeta)\,|d\zeta|
\ge\frac{1}{2}{|I_z|^\alpha}
\ge\frac{1}{2}(1-|z|^2)^\alpha.
$$
Therefore, by Fubini's theorem,
$$ 
\cD_\alpha(f)
:=\frac{1}{\pi}\int_\DD |f'(z)|^2(1-|z|^2)^\alpha\,dA(z)
\le 2\int_\TT\cD_\zeta(f)h(\zeta)\,|d\zeta|,
$$
where 
$$
\cD_\zeta(f):=\frac{1}{\pi}\int_\DD|f'(z)|^2P(z,\zeta)\,dA(z) \quad(\zeta\in\TT).
$$
Now $\cD_\zeta(f)$ is the so-called local Dirichlet of integral of $f$ at $\zeta$,
which was studied in detail by Richter and Sundberg in \cite{RS}.
In particular, they showed that, if $f$ is an outer function, then
$$
\cD_\zeta(f)
=\frac{1}{2\pi}\int_{\TT}\frac{|f^*(\zeta)|^2-|f^*(\zeta')|^2 -2|f^*(\zeta')|^2\log|f^*(\zeta)/f^*(\zeta')|}{|\zeta-\zeta'|^{2}}\,|d\zeta'|.
$$
Substituting this into the preceding estimate for $\cD_\alpha(f)$,
and noting the obvious fact that $h(\zeta)\le h(\zeta)+h(\zeta')$,
we deduce that $\cD_\alpha(f)$ is majorized by
$$
\frac{1}{\pi}\iint_{\TT}\frac{|f^*(\zeta)|^2-|f^*(\zeta')|^2 -2|f^*(\zeta')|^2\log|f^*(\zeta)/f^*(\zeta')|}{|\zeta-\zeta'|^{2}}(h(\zeta)+h(\zeta'))\,|d\zeta'|\,|d\zeta|.
$$
Exchanging the roles of $\zeta$ and $\zeta'$, we see that $\cD_\alpha(f)$ is likewise majorized by
 $$
\frac{1}{\pi}\iint_{\TT}\frac{|f^*(\zeta')|^2-|f^*(\zeta)|^2 -2|f^*(\zeta)|^2\log|f^*(\zeta')/f^*(\zeta)|}{|\zeta'-\zeta|^{2}}(h(\zeta')+h(\zeta))\,|d\zeta|\,|d\zeta'|.
$$
Taking the average of these last two estimates,
we obtain the inequality in the statement of the theorem.
\end{proof}

We are going to apply this result with $h(\zeta):=Cd(\zeta,E)^\alpha$,
where $C$ is a constant, 
$d$ denotes arclength distance on $\TT$, and $E$ is a closed subset of $\TT$.
Condition \eqref{E:hjensen} thus becomes
\begin{equation}\label{E:Kset}
\frac{1}{|I|}\int_I d(\zeta,E)^\alpha\,|d\zeta|\ge C^{-1}|I|^\alpha 
\quad\text{for all arcs~}I\subset\TT.
\end{equation} 
A set $E$ which satisfies this condition for some $\alpha,C$ is called a \emph{K-set}
(after Kotochigov).
K-sets arise as the interpolation sets for certain function spaces, 
and have several other interesting properties. 
We refer to \cite[\S1]{Br} and \cite[\S3]{Dy} for more details. 
In particular, if $E$ satisfies \eqref{E:Kset}, 
then it has measure zero and $\log d(\zeta,E)\in L^1(\TT)$.

\begin{thm}\label{T:fw} 
Let $\alpha\in(0,1)$,
let $E$ be a closed subset of $\TT$ satisfying \eqref{E:Kset},
and let $w:[0,2\pi]\to \RR^+ $ be  an increasing function 
such that  $t\mapsto \omega(t^\gamma)$ is concave for some $\gamma>2/(1-\alpha)$. 
Let $f_w$ be the outer function satisfying
$$
|f_w^*(\zeta)|=w(d (\zeta, E))\qquad \text{a.e on } \TT.
$$
Then
\begin{equation}\label{E:fw}
\cD_\alpha(f_w)\le 
C(\alpha, \gamma,E) \int_{\TT}w'(d(\zeta,E))^2 d(\zeta,E)^{1+\alpha}\,|d\zeta|.
\end{equation}
In particular $f_w\in \cD_\alpha$ if the last integral is finite.
\end{thm}

\begin{proof}
The proof is largely similar to that of \cite[Theorem~4.1]{EKR1}, 
so we give just a sketch, 
concentrating on those parts where the two proofs differ.

We begin by remarking that the concavity condition on $w$ easily implies that
$|\log w(d(\zeta,E))|\le C(w)|\log d(\zeta,E)|$, so $\log w(d(\zeta,E))\in L^1(\TT)$
and the definition of $f_w$ makes sense. 

By Theorem~\ref{T:Carleson}, we have 
$$
\cD_\alpha(f_w)\le C(\alpha,E)\iint_{\TT^2}
 \frac{(w^2(\delta)-w^2(\delta'))(\log w(\delta)-\log w(\delta'))}{|\zeta-\zeta'|^2}
(\delta^\alpha+ \delta'^\alpha )\,|d\zeta|\,|d\zeta'|,
$$
where we have written $\delta:=d(\zeta,E)$ and $\delta':=d(\zeta',E)$.

Let $(I_j)$ be the connected components of $\TT\setminus E$, and set
$$
N_E(t):=2\sum_j 1_{\{|Ij|>2t\}} \quad(0<t\le\pi).
$$
Then, for every measurable function $\Omega:[0,\pi]\to\RR^+$, we have
$$
\int_\TT \Omega(d(\zeta,E))\,|d\zeta|=\int_0^\pi \Omega(t)N_E(t)\,dt.
$$
In particular, as in \cite{EKR1}, it follows that
\begin{align*}
&\iint_{\TT^2}\frac{(w^2(\delta)-w^2(\delta'))(\log w(\delta)-\log w(\delta'))}{|\zeta-\zeta'|^2}
(\delta^\alpha+ \delta'^\alpha )\,|d\zeta|\,|d\zeta'|\\
&\le C(\alpha)\int_0^\pi\int_0^\pi
 \frac{(w^2(s+t)-w^2(t))(\log w(s+t)-\log w(t))}{s^2}
(s+t )^\alpha N_E(t)\,ds\,dt.
\end{align*}
The concavity assumption on $w$ implies that $t\to t^{1-1/\gamma}w'(t)$ is decreasing,
and thus, as in \cite{EKR1}, 
\begin{align*}
w^2(t+s)-w^2(t)\le 2\gamma w(t+s)w'(t)t\bigl((1+s/t)^{1/\gamma}-1\bigr),\\
\log w(t+s)-\log w(t)\le tw'(t)\frac{(1+s/t)^{1/\gamma}}{w(t+s)}\log(1+s/t).
\end{align*}
Combining these estimates, we obtain
\begin{align*}
&\int_0^\pi  \int_0^\pi\frac{(w^2(t+s)-w^2(t))(\log w(t+s)-\log w(t))}{s^2}(s+t)^\alpha \,ds N_E(t)\,dt\\
&\le \int_0^\pi \int_0^\pi 2\gamma w'(t)^2t^{2+\alpha}\bigl((1+s/t)^{1/\gamma}-1\bigr)(1+s/t)^{1/\gamma+\alpha}\log(1+s/t)\,\frac{ds}{s^2}N_E(t)\,dt\\
&= \int_0^\pi 2\gamma w'(t)^2t^{1+\alpha}\bigl( \int_0^{\pi/t} 2\gamma \bigl((1+x)^{1/\gamma}-1\bigr)(1+x)^{1/\gamma+\alpha}\log (1+x)\frac{dx}{x^2}\bigr)N_E(t)\,dt\\
&\le C(\alpha,\gamma) \int_0^\pi w'(t)^2t^{1+\alpha}N_E(t)\,dt.
\end{align*}
 In the last inequality we used the fact that $\gamma>2/(1-\alpha)$.
\end{proof}

\section{Generalized Cantor sets}\label{S:Cantor}

The notion of the generalized Cantor set $E$ associated to a sequence $(a_n)$ 
was defined in \S\ref{S:intro}. 
In this section we briefly describe some pertinent properties of these sets. 
We shall write
$$
\lambda_E:=\sup_{n\ge0}\frac{a_{n+1}}{a_n}.
$$
Recall that, by hypothesis, $\lambda_E<1/2$. 

Our first result shows that generalized Cantor sets satisfy \eqref{E:hjensen},
and hence that Theorem~\ref{T:fw} is  applicable to such sets.

\begin{prop} Let $E$ be a generalized Cantor set and let $\alpha\in[0,1)$.
Then, for each arc $I\subset\TT$,
$$
\frac{1}{|I|}\int_I d(\zeta,E)^\alpha\,|d\zeta|\ge C(\alpha,\lambda_E)|I|^\alpha.
$$
\end{prop}

\begin{proof} 
Let $I$ be an arc with $|I|\le 2a_0$,
and choose $n$ so that $2a_n< |I|\le 2a_{n-1}$.
Recall that the $n$-th approximation to $E$ consists of $2^n$ arcs,
each of length~$a_n$, and that the distance between these arcs is at least $a_{n-1}-2a_n$.
If $I$ meets at least two of these arcs, 
then $I\setminus E$ contains an arc $J$ of length $a_{n-1}-2a_n$, 
and if $I$ meets at most one of these arcs, 
then $I\setminus E$ contains an arc $J$ of length $(|I|-a_n)/2$. 
Thus $I\setminus E$ always contains an arc $J$ such that
$|J|/|I|\ge \min\{1/2-\lambda_E,\,1/4\}$. Consequently
$$
\frac{1}{|I|}\int_Id(\zeta,E)^\alpha\,|d\zeta|
\ge\frac{1}{|I|}\int_Jd(\zeta,E)^\alpha\,|d\zeta|
\ge\frac{1}{|I|}\frac{|J|^{\alpha+1}}{\alpha+1}
\ge C(\alpha,\lambda_E)|I|^\alpha.
$$
\end{proof}

In the next two results, we write 
$E_t:=\{\zeta\in\TT:d(\zeta,E)\le t\}$.

\begin{prop}\label{P:|Et|}
If $E$ is a generalized Cantor set,
then  $|E_t|=O(t^\mu)$ as $t\to0$, where $\mu:=1-\log2/\log(1/\lambda_E)$.
\end{prop}

\begin{proof}
Given $t\in(0,a_0]$, choose $n$ so that $a_n<t\le a_{n-1}$.
Then clearly $|E_t|\le 2^n(a_n+2t)\le 3.2^n t$.
Also $2^{n-1}=(1/\lambda_E^{n-1})^{\log2/\log(1/\lambda_E)}\le 
(a_0/t)^{\log2/\log(1/\lambda_E)}$. 
Hence $|E_t|\le C(a_0,\lambda_E)t^{1-\log2/\log(1/\lambda_E)}$.
\end{proof}

In particular, every generalized Cantor set $E$ is a 
\emph{Carleson set}, that is,   $\int_0^\pi( |E_t|/t)\,dt<\infty$.
Taylor and Williams \cite{TW} showed that
 Carleson sets are  zero sets of outer functions in $A^\infty(\DD)$. 
 This justifies a remark made in \S\ref{S:intro}.

The final property that we need concerns the $\alpha$-capacity, $C_\alpha$, 
which was defined in \S\ref{S:intro}.

\begin{thm}\label{T:capmeas}
Let $E$ be a generalized Cantor set and let $\alpha\in[0,1)$.
Then 
$$
C_\alpha(E)=0 \iff \int_0^\pi \frac{dt}{t^\alpha |E_t|}=\infty.
$$
\end{thm}

\begin{proof}
This follows easily from \cite[\S IV, Theorems 2 and 3]{Ca2}.
\end{proof}

\section{Regularization}\label{S:reg}
 
We shall need the following regularization result. The proof is the same, with minor modifications, as that of  \cite[Theorem 5.1]{EKR1}.
 
\begin{thm}\label{T:reg} 
Let $\alpha\in[0,1)$, let $\sigma\in(0,1)$ and let $a>0$.
Let  $\phi:(0,a]\to \RR^+$ be a function such that
\begin{itemize}
\item $\phi(t)/t$ is decreasing,
\item $0<\phi(t)\le t^\sigma$ for all $t\in (0,a]$,
\item $\displaystyle \int_0^a \frac{dt}{t^\alpha\phi(t)}=\infty$.
\end{itemize}
Then, given  $\rho\in(0,\sigma)$, there exists a function $\psi:(0,a]\to \RR^+ $ such that
\begin{itemize}
\item $\psi(t)/t^\rho$ is increasing,
\item $\phi(t)\le \psi(t)\le t^\sigma$ for all $t\in (0,a]$,
\item  $\displaystyle\int_{0}^{a} \frac{dt}{t^\alpha \psi(t)}=\infty$.
\end{itemize}
\end{thm}

\section{Completion of the proof of Theorem~\ref{T:suff}}\label{S:completion}
For $\alpha=0$, this theorem was proved in \cite[\S6]{EKR1}.
The proof for $\alpha\in(0,1)$ will
follow the same general lines,
and once again we shall concentrate mainly  on the places where the proofs differ.

Let $f$ be the function in  the theorem. 
Our goal is to show that $1\in[f]_{\cD_\alpha}$.

Let $E$ be as in the theorem, and let $g$ be the outer function such that
$$
|g^*(\zeta)|=d(\zeta,E)^4 \quad \text{a.e.\ on~}\TT.
$$
Then $|g(z)|\le (\pi/2)^4\dist(z,E)^4~(z\in\DD)$: 
indeed, for every $\zeta_0\in E$, we have
$$
\log|g(z)|
\le \frac{1}{2\pi}\int_\TT\frac{1-|z|^2}{|z-\zeta|^2}4\log\bigl((\pi/2)|\zeta-\zeta_0|\bigr)\,|d\zeta|
=4\log\bigl((\pi/2)|z-\zeta_0|\bigr).
$$
By assumption, the zero set  $F:=\{\zeta\in\TT:|f(\zeta)|=0\}$ in contained in $E$,
so $|g(z)|\le (\pi/2)^4\dist(z,F)^4~(z\in\DD)$.
Theorem~\ref{T:Korenblum} therefore applies,  
and we can infer that  $g\in[f]_{\cD_\alpha}$.
It thus suffices to prove that $1\in [g]_{\cD_\alpha}$.  

We shall construct a family  of functions
$w_\delta:[0,\pi]\to\RR^+$ for $0<\delta<1$ such that 
the associated  outer functions $f_{w_\delta}$ belong to 
$[g]_{\cD_\alpha}$ and satisfy:
\begin{itemize}
\item[(i)] $|f_{w_\delta}^*|\to 1$ a.e.\ on $\TT$ as $\delta\to0$,
\item[(ii)] $|f_{w_\delta}(0)|\to 1$ as $\delta\to0$,
\item[(iii)] $\liminf_{\delta\to0}\|f_{w_\delta}\|_\alpha<\infty$.
\end{itemize}
If such a family exists, then a subsequence of the $f_{w_\delta}$ converges weakly to $1$ in $\cD_\alpha$, and since they all belong to $[g]_{\cD_\alpha}$, it follows that $1\in [g]_{\cD_\alpha}$, as desired.

By Proposition~\ref{P:|Et|}, there exists $\mu>0$ such that $|E_t|=O(t^\mu)$ as $t\to0$.
Fix $\rho,\sigma$ satisfying
$$
\frac{1-\alpha}{2}<\rho<\sigma<\min\Bigl\{1-\alpha,~\frac{1-\alpha+\mu}{2}\Bigr\}.
$$
Define $\phi:(0,\pi]\to\RR^+$ by
$$
\phi(t):=\max\Bigl\{\min\{|E_t|,~t^\sigma\},~t^{1-\alpha}\Bigr\} \quad(t\in(0,\pi]).
$$
Clearly $\phi(t)/t$ is increasing and $0\le \phi(t)\le t^\sigma$ for all $t$.
We claim also that
\begin{equation}\label{E:claim}
\int_0^\pi \frac{dt}{t^\alpha \phi(t)}=\infty.
\end{equation}
To see this, note that
$$
\int_t^\pi \frac{ds}{s^\alpha|E_s|}\ge 
\frac{t}{|E_t|}\int_t^\pi\frac{ds}{s^{\alpha+1}}
=C(\alpha)\frac{t^{1-\alpha}}{|E_t|},
$$
whence
\begin{align*}
\int_\epsilon^\pi \frac{dt}{t^\alpha\phi(t)}
&\ge \int_\epsilon^\pi \frac{dt}{\max\{t,t^\alpha|E_t|\}}\\
&\ge C(\alpha)\int_\epsilon^\pi \frac{ds}{t^\alpha|E_t|(\int_t^\pi ds/s^\alpha|E_s|)}\\
&\ge C(\alpha)\log \int_\epsilon^\pi \frac{dt}{t^\alpha|E_t|}.
\end{align*}
Since $E$ is a generalized Cantor set of $\alpha$-capacity zero, 
Theorem~\ref{T:capmeas} shows that
$$
\int_0^\pi \frac{dt}{t^\alpha|E_t|}=\infty.
$$
Consequently \eqref{E:claim} holds, as claimed.

We have now shown that $\phi$ satisfies all the hypotheses of the regularization theorem,
Theorem~\ref{T:reg}. Therefore there exists a function $\psi:(0,\pi]\to\RR^+$ satisfying the conclusions of that theorem, namely: $\psi(t)/t^\rho$ is increasing, $t^{1-\alpha}\le\phi(t)\le \psi(t)\le t^\sigma$ for all $t$, and $\int_0^\pi dt/t^\alpha\psi(t)=\infty$.

For $0<\delta<1$, we define $w_\delta:[0,\pi]\to\RR^+$ by
$$
w_\delta(t):=
\begin{cases}
\displaystyle \frac{\delta^\rho}{\psi(\delta)}t^{1-\alpha-\rho} &0\le t\le\delta\\
\displaystyle A_\delta-\log\int_t^\pi\frac{ds}{s^\alpha\psi(s)} &\delta<t\le\eta_\delta\\
\displaystyle1 &\eta_\delta<t\le\pi,
\end{cases}
$$
where $A_\delta$ and $\eta_\delta$ are constants chosen to make $w_\delta$ continuous.

Let us show that $f_{w_\delta}\in[g]_{\cD_\alpha}$. Note first that $w_\delta(t)/t^{1-\alpha-\rho}$ is a bounded function. Therefore $f_{w_\delta}/g^{(1-\alpha-\rho)/4}$ is bounded. Using Theorem~\ref{T:fw}, we have $g^{(1-\alpha-\rho)/4}\in\cD_\alpha$. Consequently, by a theorem of Aleman \cite[Lemma~3.1]{Al}, $f_{w_\delta}\in[g^{(1-\alpha-\rho)/4}]_{\cD_\alpha}$. Using another result of Aleman \cite[Theorem~2.1]{Al}, we have
$g^{(1-\alpha-\rho)/4}\in[g]_{\cD_\alpha}$. Hence $f_{w_\delta}\in[g]_{\cD_\alpha}$, as claimed.

It remains to check that the functions $f_{w_\delta}$ satisfy properties (i)--(iii) above.
The verifications run along   the same lines as  those in \cite[\S6]{EKR1}, 
using the properties of $\psi$ above, and  Theorem~\ref{T:fw} in place of \cite[Theorem~4.1]{EKR1}.

\end{document}